\documentclass[a4paper, 12pt]{article}

\usepackage{amsfonts, amsmath, amsthm}
\usepackage{amssymb}
\usepackage{latexsym}
\usepackage[dvips]{graphicx}
\usepackage{color}

\theoremstyle{plain}
\newtheorem{thm}{Theorem}[section]
\newtheorem{prp}[thm]{Proposition}
\newtheorem{lem}[thm]{Lemma}
\newtheorem{cor}[thm]{Corollary}

\newtheorem{rmk}{Remark}[section]

\numberwithin{equation}{section}

\newcommand{\R}{\mathbb{R}}
\newcommand{\C}{\mathbb{C}}

\newcommand{\pa}{\partial}
\newcommand{\eps}{\varepsilon}

\newcommand{\uT}{\underline{T}}

\newcommand{\op}[1]{\mathcal{#1}}

\DeclareMathOperator{\imagpart}{\rm Im}

\newcommand{\dis}{\displaystyle}

\begin{document}
\title{\Large 
A sharp lower bound for the lifespan of small solutions to the 
Schr\"odinger equation with a subcritical power nonlinearity 
\\
}

\author{
          Yuji Sagawa \thanks{
              Department of Mathematics, Graduate School of Science, 
              Osaka University. 
              1-1 Machikaneyama-cho, Toyonaka, Osaka 560-0043, Japan. 
              E-mail: {\tt y-sagawa@cr.math.sci.osaka-u.ac.jp}
             }
           \and  
          Hideaki Sunagawa \thanks{
              E-mail: {\tt sunagawa@math.sci.osaka-u.ac.jp}
             }
           \and  
          Shunsuke Yasuda \thanks{
              E-mail: {\tt s-yasuda@cr.math.sci.osaka-u.ac.jp}
             }
}

\date{\today }   
\maketitle

\noindent{\bf Abstract:}\ 
Let  $T_{\eps}$ be the lifespan for the solution 
to the Schr\"odinger equation on $\R^d$ with a power nonlinearity 
$\lambda |u|^{2\theta/d}u$ ($\lambda \in \C$, $0<\theta<1$) and the initial 
data in the form $\eps \varphi(x)$. We provide a sharp lower bound estimate 
for $T_{\eps}$ as $\eps \to +0$ which can be written explicitly by 
$\lambda$, $d$, $\theta$, $\varphi$ and $\eps$. 
This is an improvement of the previous result by H.~Sasaki 
[Adv. Diff. Eq. {\bf 14} (2009), 1021--1039].
\\

\\


\section{Introduction}  \label{sec_intro}
\subsection{Backgrounds}  \label{subsec_backgrounds}
We consider the following initial value problem: 
\begin{align}
\left\{\begin{array}{cl}
 \dis{i\pa_t u+\frac{1}{2}\pa_x^2 u=\lambda |u|^{p-1}u}, & t>0, \ x\in \R,\\
 u(0,x)=\eps \varphi(x), &x\in \R,
\end{array}\right.
\label{nls}
\end{align}
where $i=\sqrt{-1}$, $\lambda\in \C$ and $p>1$. 
$\varphi$ is a prescribed $\C$-valued function which belongs to 
a suitable weighted Sobolev space, and $\eps>0$ is a small parameter 
which is responsible for the size of the initial data. 
We are interested in the lifespan $T_{\eps}$ for the solution $u=u(t,x)$ 
to \eqref{nls} in the case of $p<3$ and $\imagpart \lambda>0$. 
Before going into details, let us summarize the backgrounds briefly. 

First we consider the simpler case $p>3$. In this case, small data 
global existence for \eqref{nls} is well-known. Moreover, the solution behaves 
like the free solution in the large time. On the other hand, when $p\le 3$, 
non-existence of asymptotically free solution has been shown in 
\cite{Str}, \cite{Ba}. 
Roughly speaking, the critical exponent $p=3$ comes from the condition for 
convergence of the integral 
\[
 \int_1^{\infty} \frac{dt}{t^{(p-1)/2}}.
\]
Note that this threshold becomes $p=1+2/d$ in the $d$-dimensional settings. 
Next let us turn our attention to the case $p\le 3$. 
In \cite{HN}, it has been shown that the solution to \eqref{nls} with 
$p=3$ and $\lambda\in \R$ behaves like
\[
 u(t,x)=\frac{1}{\sqrt{it}} \alpha(x/t) 
 e^{i\{x^2/(2t)  - \lambda |\alpha(x/t)|^2 \log t \}}
 +o(t^{-1/2})
\quad \mbox{in}\ L^{\infty}(\R_x)
\]
as $t\to \infty$ with a suitable $\C$-valued function $\alpha$ satisfying 
$\|\alpha\|_{L^{\infty}} \le C\eps$. 
An important consequence of this asymptotic expression
is that the solution decays like $O(t^{-1/2})$ in $L^{\infty}(\R_x)$, 
while it does not behave like the free solution unless $\lambda = 0$. 
In other words, the additional logarithmic factor in the phase reflects the
long-range character of the cubic nonlinear Schr\"odinger equations in one 
space dimension.
This result has been extended in \cite{HKN} to the case where 
$p$ is less than and sufficiently close to $3$. 
When $\lambda\in \C$, the situation changes slightly. 
Indeed, it has been verified in \cite{Shim} that the
small data solution to \eqref{nls} decays like $O(t^{-1/2}(\log t)^{-1/2})$ 
in $L^{\infty}(\R_x)$  as $t \to \infty$ if $p=3$ and $\imagpart \lambda <0$. 
This gain of additional logarithmic time decay should be interpreted as 
another kind of long-range effect (see also \cite{Su2} for a closely related 
result for the Klein-Gordon equation). The above-mentioned result has been 
extended in \cite{KitaShim1}, \cite{KitaShim2}, \cite{HLN}, \cite{JJL}, etc., 
to the case $p<3$ and $\imagpart \lambda <0$. 
However, it should be noted that these results essentially rely on 
the a priori $L^2$-bound for the solution $u$ coming from the conservation 
law
\[
  \|u(t,\cdot)\|_{L^2}^2 
  - 2\imagpart \lambda \int_0^t\|u(\tau,\cdot)\|_{L^{p+1}}^{p+1}d\tau
=
\|u(0,\cdot)\|_{L^2}^2, 
\]
which is valid only when $\imagpart \lambda \le 0$. 
In what follows, we focus on the remaining case  $p\le 3$ and 
$\imagpart \lambda>0$. This is the worst situation for global existence 
because the nonlinearity must be considered as a long-range perturbation and 
the a priori $L^2$-bound for $u$ is violated. 
To the authors' best knowledge, there is no positive result in that case.
As for the lifespan $T_{\eps}$, the standard perturbative argument 
yields a lower estimate in the form 
\[
T_{\eps} \geq
\left\{\begin{array}{cl}
e^{C/\eps^{2}} & (\mbox{when}\ p=3) \\
C\eps^{-2(p-1)/(3-p)} & (\mbox{when}\ 1<p<3)
\end{array}\right.
\]
with some $C>0$, provided that $\eps$ is suitably small 
(see Section~\ref{sec_rough} below for more detail). 
In other words, we have
\[
 \liminf_{\eps \to +0} 
 \int_1^{T_{\eps}}\left(\frac{\eps}{t^{1/2}}\right)^{p-1} dt
 >0.
\]
However, this estimate does not tell us the dependence of $T_{\eps}$ 
on $\imagpart \lambda$. So we are led to the question: 
{\em how does $T_{\eps}$ depend on $\imagpart \lambda$?} 
In the cubic case, two of the authors have derived 
the following more precise estimate for $T_{\eps}$ in 
the previous papers \cite{Su}, \cite{SaSu}: 
\[
 \liminf_{\eps \to +0} (\eps^2 \log T_{\eps})
 \ge 
 \frac{1}{\dis{2\imagpart \lambda \sup_{\xi \in \R} |\hat{\varphi}(\xi)|^2}},
\]
where
\[
 \hat{\varphi}(\xi)
 = 
 \frac{1}{\sqrt{2\pi}} \int_{\R} e^{-iy\xi} \varphi(y)\, dy,
\quad \xi \in \R.
\]
This gives an answer to the question raised above for the cubic case. 
In fact, more general cubic nonlinear terms depending also on $\pa_x u$ have 
been treated in \cite{Su}, \cite{SaSu} (see also \cite{MP} for a 
related work). 
When $p<3$ and $\imagpart \lambda>0$, the situation is the most delicate and 
quite little is known so far. To the authors' knowledge, 
there is only one result which concerns the dependence of $T_{\eps}$ 
on $\imagpart \lambda$ in the case of $p<3$: 
\begin{prp}[Sasaki~\cite{Sas}] \label{prp_prev} 
Assume $2\le p<3$, $\imagpart \lambda >0$ and $(1+x^2)\varphi \in \Sigma$. 
Let $T_{\eps}$ be the supremum of $T>0$ such that \eqref{nls} admits a 
unique solution $u\in C([0,T); \Sigma)$. Then we have
\[
 \liminf_{\eps\to +0} \left( \eps^{2(p-1)/(3-p)} {T_{\eps}}\right)
 \ge 
 \left(\frac{3-p}
 {\dis{2(p-1)\imagpart \lambda \sup_{\xi \in \R} 
   |\hat{\varphi}(\xi)|^{p-1} }}\right)^{2/(3-p)},
\]
where 
$\Sigma=\{f\in L^2(\R)\, |\, \|f\|_{\Sigma}<\infty\}$ with 
$\|f\|_{\Sigma}=\|f\|_{L^2}+\|\pa_x f\|_{L^2}+\|xf\|_{L^2}$.
\end{prp}

The aim of this paper is to improve Proposition~\ref{prp_prev} regarding 
the following three points: 
\begin{itemize}
\item
 to extend the admissible value of $p$ to the full range $1<p<3$, 
\item
 to relax the decay assumption on $\varphi(x)$ as $|x|\to \infty$, 
\item
 to give a higher dimensional generalization. 
\end{itemize}

\subsection{Main result}  \label{sec_main}
In what follows, we consider a $d$-dimensional generalization of 
\eqref{nls}. For the notational convenience, we write the power $p$ 
of the nonlinearity as $p=1+2\theta/d$ so that the condition $1<p<1+2/d$ 
is interpreted as $0<\theta <1$. 
Then we are led to the following initial 
value problem: 
\begin{align}
\left\{\begin{array}{cl}
 \dis{i\pa_t u+\frac{1}{2}\Delta u=\lambda |u|^{2\theta/d}u}, 
 & t>0, \ x\in \R^d,\\
 u(0,x)=\eps \varphi(x), &x\in \R^d,
\end{array}\right.
\label{nls_d}
\end{align}
where $\Delta=(\pa/\pa x_1)^2+\cdots+(\pa/\pa x_d)^2$  for 
$x=(x_1,\ldots,x_d)\in \R^d$. To state the main result, let us introduce 
some notations. For $s$, $\sigma\ge 0$, we denote by $H^{s,\sigma}$ 
the weighted Sobolev spaces  
\[
 H^{s,\sigma}:=\bigl\{f\in L^2(\R^d)\, \bigm|\,  
(1+|x|^2)^{\sigma/2} (1-\Delta)^{s/2} f \in L^2(\R^d) \bigr\}
\]
equipped with the norm 
\[
\|f\|_{H^{s,\sigma}}:=\|(1+|x|^2)^{\sigma/2} (1-\Delta)^{s/2} f\|_{L^2}. 
\]
We also define $\Sigma^{s}:=H^{s,0}\cap H^{0,s}$ with the norm 
$\|f\|_{\Sigma^{s}}:=\|f\|_{H^{s,0}}+\|f\|_{H^{0,s}}$. 
We set $\op{U}(t):=\exp(\frac{it}{2}\Delta)$ so that the solution $v$ 
to the free Schr\"odinger equation 
\[i\pa_t v+\frac{1}{2}\Delta v=0, \quad 
v(0,x)=\phi(x)
\] 
can be written as $v(t)=\op{U}(t)\phi$. 
The main result is as follows. 
\begin{thm} \label{thm_main}
Let $1\le d\le 3$, 
$0<\theta<1$ and $\lambda\in \C$ with $\imagpart \lambda >0$. 
Assume 
\begin{align}
d/2<s<\min\{2,1+2\theta/d\}
\label{index}
\end{align} 
and $\varphi \in \Sigma^{s}$. 
Let $T_{\eps}$ be the supremum of $T>0$ such that \eqref{nls_d} admits a 
unique solution $u$ satisfying  $\op{U}(\cdot)^{-1}u\in C([0,T); \Sigma^{s})$. 
Then we have
\begin{align}
 \liminf_{\eps\to +0} \left( \eps^{2\theta/d} {T_{\eps}}^{1-\theta}\right)
 \ge 
 \frac{(1-\theta)d}
 {\dis{2\theta \imagpart \lambda \sup_{\xi \in \R^d} 
   |\hat{\varphi}(\xi)|^{2\theta/d}}},
\label{goal}
\end{align}
where 
\[
 \hat{\varphi}(\xi)=\op{F}\varphi(\xi)
 := 
 \frac{1}{(2\pi)^{d/2}} \int_{\R^d} e^{-iy\cdot \xi} \varphi(y)\, dy,
\quad \xi \in \R^d.
\]
\end{thm}

\begin{rmk}
The assumption \eqref{index} is never satisfied when $d\ge 4$. 
That is the reason why Theorem~\ref{thm_main} is available only for $d\le 3$. 
When $d=1$ or $2$, \eqref{index} is satisfied for any $0<\theta <1$. 
In particular, our result can be viewed as an extension of Proposition 
\ref{prp_prev} 
because it  corresponds to the case of 
$d=1$, $1/2\le \theta <1$ and $s=1$ in Theorem~\ref{thm_main}. 
On the other hand, when $d=3$, \eqref{index} is satisfied only if 
$\theta> 3/4$ (or, equivalently, $3/2<p<5/3$ with $p=1+2\theta/3$). 
The authors do not know whether the same assertion holds true or not 
when $d\ge 4$ or $d=3$ with $\theta\le 3/4$.
\end{rmk}

\begin{rmk}
The authors do not know whether \eqref{goal} is optimal or not. 
An example of blowing-up solution to \eqref{nls} with arbitrarily small 
$\eps>0$ has been given by Kita \cite{Kita} under a particular choice of 
$\varphi$ and some additional restrictions on $\lambda$ and $p$. However, 
it seems difficult to specify the lifespan for the blowing-up solution 
given in \cite{Kita}.
\end{rmk}

Now, let us explain the differences between the approach of \cite{Sas} and 
ours. The method of \cite{Sas} consists of two steps: the first is to 
construct a suitable approximate solution $u_a$ which blows up at the expected 
time, and the second is to get an a priori estimate not for the solution $u$ 
itself but for their difference $u-u_a$ 
(see also \cite{Su} for the cubic case). 
Drawbacks of this approach come from the first step. 
In fact, according to Remark 1.3 in \cite{Sas}, this approach can not be 
used in the case $1<p<2$. Remark that this implies the method of 
\cite{Sas} is not suitable for $d$-dimensional settings when $d\ge 2$, 
because our main interest is the case of $p<1+2/d$. Also, in view of 
Proposition~3.1 in \cite{Sas}, the additional decay assumption on $\varphi$ 
as $|x|\to \infty$ (i.e., higher regularity for $\hat{\varphi}$) 
seems essential for the method of \cite{Sas}. 
On the other hand, our approach presented below does not rely on 
approximate solutions at all. Instead, 
we will reduce the original PDE \eqref{nls_d} to a simpler ordinary 
differential equation satisfied by 
$A(t,\xi)=\op{F}\bigl[\op{U}(t)^{-1} u(t,\cdot)\bigr](\xi)$ 
up to a harmless remainder term $R$ (see \eqref{reduced_ODE} below). An ODE 
lemma prepared in Section~\ref{sec_prelim} below will allow us to get 
an a priori bound for $u$ directly. Similar idea has been used in \cite{SaSu} 
for one-dimensional cubic derivative nonlinear Schr\"odinger equations, 
but we must be more careful because we are considering the situation in which 
the degree of the nonlinearity is lower.

We close the introduction with the contents of this paper. 
In the next section, we state basic lemmas which will be useful in the 
subsequent sections. In Section~\ref{sec_rough}, we will derive a rough 
lower estimate for $T_{\eps}$, that is, 
$\dis{\liminf_{\eps \to +0} (\eps^{2\theta/d} T_{\eps}^{1-\theta})>0}$. 
Section~\ref{sec_prelim} is devoted to an ODE lemma which plays an important 
role in getting an a priori bound for the solution. 
After that, the main theorem will be proved in Section~\ref{sec_bootstrap} 
by means of the so-called bootstrap argument. Finally, in 
Section~\ref{sec_critical}, we give a few comments on the critical case 
$\theta =1$. 
In what follows, we denote several positive constants by the same letter C, 
which may vary from one line to another.

\section{Basic lemmas}  \label{sec_basic}

In this section, we introduce several lemmas that 
will be useful in the subsequent sections.

\begin{lem} \label{lem_wSob}
Let $s>d/2$. 
There exists a constant $C$ such that
\[
\| \phi\|_{L^{\infty}}
\le
\frac{C}{(1+t)^{d/2}} 
\| \mathcal{U}(t)^{-1} \phi\|_{\Sigma^s}
\]
for $t\ge 0$.
\end{lem}

\begin{proof}
We start with the standard Gagliardo-Nirenberg-Sobolev inequality: 
\[
\|\phi\|_{L^{\infty}}
\le 
C \|\phi\|_{L^2}^{1-d/2s}\|(-\Delta)^{s/2} \phi\|_{L^2}^{d/2s}.
\]
We also introduce 
$\op{M}(t)=\exp(\frac{i|x|^2}{2t})$. Then we can check that 
\begin{align*}
\op{U}(t)|x|^s\op{U}(t)^{-1} \phi
=
\op{M}(t) (-t^2\Delta)^{s/2}\op{M}(t)^{-1}\phi,
\end{align*}
from which it follows that 
\begin{align*}
t^{d/2}\|\phi\|_{L^{\infty}}
&=
t^{d/2}\|\op{M}(t)^{-1}\phi\|_{L^{\infty}}\\
&\le 
C\|\op{M}(t)^{-1}\phi\|_{L^2}^{1-d/2s}
 \|(-t^2 \Delta)^{s/2} \op{M}(t)^{-1}\phi\|_{L^2}^{d/2s}\\
&\le 
C\|\phi\|_{L^2}^{1-d/2s}
 \||x|^s\op{U}(t)^{-1}\phi\|_{L^{2}}^{d/2s}.
\end{align*}
Combining the two inequalities above, we obtain 
\begin{align*}
 (1+t)^{d/2}\|\phi\|_{L^{\infty}}
 &\le 
 \frac{C(1+t)^{d/2}}{(1+t^{d/2})} \|\phi\|_{L^2}^{1-d/2s}
 \bigl( \|(-\Delta)^{s/2} \phi\|_{L^2}^{d/2s}
 + \bigl\| |x|^s\op{U}(t)^{-1}\phi \bigr\|_{L^2}^{d/2s} 
 \bigr)\\
 &\le
 C
\left(
\|\phi \|_{H^{s,0}}+ \| \mathcal{U}(t)^{-1} \phi\|_{H^{0,s}}
\right)\\
 &=
 C \| \mathcal{U}(t)^{-1} \phi\|_{\Sigma^{s}}.
\end{align*}

\end{proof}

\begin{lem} \label{lem_Linfty}
Let $\gamma \in (0,1]$ and $s>d/2+2\gamma$. 
There exists a constant $C$ such that
$$
\| \phi\|_{L^{\infty}}
\le
\frac{1}{t^{d/2}} \|\mathcal{F} \mathcal{U}(t)^{-1} \phi \|_{L^{\infty}}
+ 
 \frac{C}{t^{d/2+\gamma}} \| \mathcal{U}(t)^{-1} \phi\|_{H^{0,s}}
$$ 
for $t\geq 1$.
\end{lem}
See Lemma 2.2 in \cite{HN} for the proof.

Next we define $G_p:\C\to \C$ with $p>1$ by 
$G_p(z)=|z|^{p-1} z$ for $z\in \C$.
Note that the nonlinear term in \eqref{nls_d} is 
$\lambda G_{1+2\theta/d}(u)$ 
with $0<\theta<1$, $\imagpart \lambda>0$. 
The following lemmas are concerned with estimates for $G_p$:

\begin{lem} \label{lem_elementary}
For $z, w\in \C$, we have 
\[
\bigl| G_p(z) -G_p(w) \bigr|
\le 
p (|z|+|w|)^{p-1} |z-w|.
\]
\end{lem}
\begin{proof} Without loss of generality, we may assume $|z|>|w|$. 
For $\nu >0$, we observe the relations
\[
 |z|^\nu -|w|^\nu
 =
 (|z|-|w|) 
 \int_{0}^{1} \nu \bigl(t |z|+(1-t)|w| \bigr)^{\nu-1}\, dt
\]
and
\[
\sup_{t\in [0,1]}\bigl(t |z|+(1-t)|w| \bigr)^{\nu-1}|w|
\le
 \left\{\begin{array}{cl}
 (|z|+|w|)^{\nu -1}|w| & \mbox{ (if $\nu \ge 1$)}\\[3mm]
 |w|^{\nu} & \mbox{ (if $\nu < 1$)}
\end{array}\right\}
 \le 
 (|z|+|w|)^{\nu}.
\]
Then we have 
\[
 \bigl| (|z|^{\nu} -|w|^{\nu} )w \bigr|
 \le
 \bigl| |z|-|w|\bigr|\cdot \nu \bigl(|z|+|w| \bigr)^{\nu} 
 \le
 \nu \bigl(|z|+|w| \bigr)^{\nu} |z-w|.
\]
We apply the above inequality with $\nu=p-1$ to obtain
\[
 \bigl|G_p(z) -G_p(w) \bigr|
 \le  \bigl|(|z|^{p-1} - |w|^{p-1})w \bigr| + |z|^{p-1} |z-w|
 \le 
 p\bigl(|z|+|w| \bigr)^{p-1} |z-w|.
\]
\end{proof}

\begin{lem} \label{lem_composition} 
Let $0\le s < \min\{2,p\}$. There exists a constant $C$ such that
$$
 \|G_p(\phi)\|_{H^{s,0}} 
 \le 
 C\|\phi\|_{L^{\infty}}^{p-1} \|\phi\|_{H^{s,0}}
$$
and
$$
 \|\mathcal{U}(t)^{-1} G_p(\phi)\|_{H^{0,s}} 
 \le 
 C\|\phi\|_{L^{\infty}}^{p-1} \|\mathcal{U}(t)^{-1} \phi\|_{H^{0,s}}
$$
for $t\ge 0$.
\end{lem}
For the proof, see Lemma 3.4 in \cite{GOV},  Lemma 2.3 in \cite{HN}, etc.

\begin{cor} \label{cor_composition} 
Let $d/2 < s < \min\{2,p\}$. There exists a constant $C$ such that
\[
  \|\mathcal{U}(t)^{-1} G_p(\phi)\|_{\Sigma^{s}} 
\le
 \frac{C}{(1+t)^{d(p-1)/2}} \|\mathcal{U}(t)^{-1} \phi\|_{\Sigma^{s}}^{p} 
\]
for $t\ge 0$. 
\end{cor}
\begin{proof} 
By  Lemmas~\ref{lem_composition}  and \ref{lem_wSob}, we have
\begin{align*}
  \|\mathcal{U}(t)^{-1} G_p(\phi)\|_{\Sigma^{s}} 
 &= 
\|G_p(\phi)\|_{H^{s,0}} 
 + \bigl\| \op{U}(t)^{-1}G_p(\phi) \bigr\|_{H^{0,s}}\\
 &\le 
 C\|\phi\|_{L^{\infty}}^{p-1} \bigl(\|\phi\|_{H^{s,0}}
 +\|\op{U}(t)^{-1}\phi\|_{H^{0,s}}\bigr)\\
 &\le 
 \frac{C}{(1+t)^{d(p-1)/2}} \|\mathcal{U}(t)^{-1} \phi\|_{\Sigma^{s}}^{p}. 
\end{align*}
\end{proof}

\begin{lem} \label{lem_remainder}
Let $\gamma\in (0,1/2]$ and $d/2+2\gamma<s<\min\{2,p\}$. 
Then there exists a constant $C$ such that
\[
 \left\|
   \mathcal{F} \mathcal{U}(t)^{-1} G_p(\phi)
   -
   \frac{1}{t^{d(p-1)/2}} G_p\bigl(\mathcal{F} \mathcal{U}(t)^{-1}\phi\bigr)
 \right\|_{L^{\infty}} 
 \le 
 \frac{C}{t^{d(p-1)/2+\gamma}} \|\mathcal{U}(t)^{-1} \phi \|_{H^{0,s}}^p
\]
for $t\geq 1$.
\end{lem}
This lemma can be shown in almost the same way as the derivation of (3.16) and 
(3.17) in \cite{HN} (see also Lemma 2.2 in \cite{KitaShim1}), so we skip the 
proof.

\section{A rough lower estimate for the lifespan}  \label{sec_rough}
In what follows, we write 
$N(u)=\lambda |u|^{2\theta/d}u=\lambda G_{1+2\theta/d}(u)$ 
and 
$\Phi=\|\varphi\|_{\Sigma^s}$, where $s$ satisfies \eqref{index}. 
The goal of this section is to derive a rough lower estimate for $T_{\eps}$.  
The argument of this section is quite standard and any 
new idea is not needed, so we shall be brief.

\begin{prp} \label{prp_rough}
Let $T_{\eps}$ be the lifespan defined in the statement of 
Theorem~\ref{thm_main}. There exists $D_0>0$ such 
that $T_{\eps}\ge D_0 \eps^{-2\theta/(1-\theta)d}$. 
Moreover the solution $u$ satisfies 
\begin{align}
 \|\op{U}(t)^{-1} u(t)\|_{\Sigma^{s}} \le 2\Phi \eps
\label{est_rough}
\end{align}
for $t\le D_0\eps^{-2\theta/(1-\theta)d}$. 
\end{prp}

\begin{proof} 
Since the local existence in $\Sigma^s$ 
is well-known (see e.g., \cite{Ca} and the references cited 
therein), 
what we have to do is to see the solution $u(t)$ stays bounded as long as 
$t$ is less than the expected value. 

Let $T>0$ and let $u(t)$ be the solution to \eqref{nls_d} in the time 
interval $[0,T)$. We set 
\[
 E(T)=\sup_{t\in[0,T)} \|\op{U}(t)^{-1} u(t)\|_{\Sigma^{s}}.
\]
Then, it follows form Corollary \ref{cor_composition} that 
\[
 \bigl\| \op{U}(t)^{-1}N(u) \bigr\|_{\Sigma^{s}}
 \le 
 \frac{CE(T)^{2\theta/d+1}}{(1+t)^{\theta}}
\]
for $t<T$. Therefore the standard energy integral method leads to 
\begin{align*}
 E(T)
&\le
 \|u(0)\|_{\Sigma^s} 
+ C\int_0^T \bigl\| \op{U}(t)^{-1}N(u) \bigr\|_{\Sigma^{s}}dt\\
&\le
 \eps \|\varphi\|_{\Sigma^s}
 + CE(T)^{2\theta/d+1}\int_0^T \frac{dt}{(1+t)^{\theta}}\\
 &\le
 \Phi \eps +C_*E(T)^{2\theta/d +1} T^{1-\theta},
\end{align*}
where the constant $C_*$ is independent of $\eps$ and $T$. 
With this $C_*$, we choose $D_0>0$ so that 
\[
 C_* 3^{1+2\theta/d} \Phi^{2\theta/d} {D_0}^{1-\theta}\le 1.
\]
Now we assume $E(T)\le 3\Phi\eps$. Then the above estimate yields 
\begin{align*}
 E(T)
\le 
 \Phi\eps + 
C_*(3\Phi\eps)^{2\theta/d +1} (D_0 \eps^{-2\theta/d(1-\theta)})^{1-\theta}
\le 
 2\Phi \eps
\end{align*}
if $T\le D_0 \eps^{-2\theta/d(1-\theta)}$. This shows that the solution $u(t)$ 
can exist as long as $t\le D_0 \eps^{-2\theta/d(1-\theta)}$.
In other words, we have $T_\eps\ge D_0 \eps^{-2\theta/d(1-\theta)}$. 
We also have the desired estimate \eqref{est_rough}. 
 \end{proof}

\begin{rmk}
In the proof of Proposition \ref{prp_rough}, we do not use any information on 
the sign of $\imagpart \lambda$.  
We need something more to clarify the dependence of $T_\eps$ on 
$\imagpart \lambda$, that is our main purpose of the present work.
\end{rmk}

\section{An ODE Lemma}  \label{sec_prelim}
In this section, we introduce an ODE lemma which will be used 
effectively in the next section.
The argument in this section is a modification of  that of \S 2 in 
\cite{SaSu} to fit for the present purpose.

Throughout this section, we always suppose $0<a<1$, $b>0$ and  
$\lambda\in \C$ with $\imagpart \lambda>0$. Let $\psi_0:\R^d\to \C$ 
be a continuous function satisfying
\[
\Psi_0 :=\sup_{\xi \in \R^d} |\psi_0(\xi)|<\infty.
\]
We set $q=\frac{b}{2(1-a)}$ and define $\tau_1 >0$ by 
\[
 \frac{1}{\tau_1}:= 
 \bigl(2q \imagpart \lambda {\Psi_0}^{b} \bigr)^{1/(1-a)}.
\]
For fixed $t_*>0$, let $\eta_0(t,\xi)$ be the solution to 
\begin{align}
\left\{\begin{array}{cl}
 \dis{i\pa_t \eta_0=\frac{\lambda}{t^{a}} |\eta_0|^{b}\eta_0}, 
 & t>t_*,\, \xi \in \R^d, \\
 \eta_0(t_*,\xi)=\eps \psi_0(\xi), & \xi \in \R^d,
\end{array}\right.
\label{ode_unperturbed}
\end{align}
where $\eps>0$ is a parameter. It is immediate to check that 
\[
 |\eta_0(t,\xi)|^{b}
 =
 \frac{(\eps |\psi_0(\xi)|)^{b}}
      {1+2q \imagpart \lambda |\psi_0(\xi)|^{b} \eps^b t_*^{1-a} 
 - 2q \imagpart \lambda |\psi_0(\xi)|^{b} \eps^{b} t^{1-a}}
\]
as long as the denominator is strictly positive. 
In view of this expression, we see that 
\begin{align}
 \sup_{(t,\xi)\in [t_*, \sigma \eps^{-2q}]\times \R^d}
 |\eta_0(t,\xi)| \le C_0 \eps
\label{est_eta_0}
\end{align}
for $\sigma \in (0, \tau_1)$, where 
\[C_0=\frac{\Psi_0}{\bigl(1-(\sigma/\tau_1)^{1-a} \bigr)^{1/b}}. 
\]
Next we consider a perturbation of \eqref{ode_unperturbed}. 
Let  $\uT>t_*$ and 
let $\psi_1:\R^d\to \C$, $\rho:[t_*, \uT)\times \R^d \to \C$ 
be continuous functions satisfying
\[
 |\psi_1(\xi)|\le C_1 \eps^{1+\delta}
\]
and 
\[
 |\rho(t,\xi)| \le \frac{C_2 \eps^{1+b+\delta}}{t^{a}}
\]
with some positive constants $C_1$, $C_2$ and $\delta>0$.
Let $\eta(t,\xi)$ be the solution to 
\begin{align*}
\left\{\begin{array}{cl}
 \dis{i\pa_t \eta=\frac{\lambda}{t^{a}} |\eta|^{b}\eta+\rho }, 
 & t\in (t_*, \uT),\ \xi \in \R^d, \\
 \eta(t_*,\xi)=\eps \psi_0(\xi)+\psi_1(\xi), & \xi \in \R^d.
\end{array}\right.
\end{align*}
The following lemma asserts that an estimate similar to 
\eqref{est_eta_0} remains valid if \eqref{ode_unperturbed} 
is perturbed by $\rho$ and $\psi_1$: 
\begin{lem} \label{lem_ODE} 
Let $\sigma \in (0,\tau_1)$ and let $\eta(t,\xi)$ be as above. 
We set $T_*=\min\{\uT,\sigma \eps^{-2q}\}$ for 
$0<\eps \le \min\{1, \sigma^{-1/q}, M^{-1/\delta} \}$. 
We have 
\[
 |\eta(t,\xi)| \le C_0\eps + M \eps^{1+\delta}
 \le (C_0+1)\eps
 \]
for $(t,\xi) \in [t_*,T_*)\times \R^d$, where 
\[
 M=2\left(C_1^2 +\frac{C_2^2}{2C_3}\right)^{1/2} 
   \exp\left( \frac{C_3\sigma^{1-a}}{2(1-a)} \right)
\]
with
\[
C_3=
2|\lambda|(b+1) (2C_0 +1)^{b} +\frac{1}{2}.
\]

\end{lem}

\begin{proof}
We set $w=\eta-\eta_0$ and 
\[
 T_{**}=\sup\Bigl\{ \tilde{T} \in [t_*, T_*)\, 
 \Bigm|\, 
   \sup_{(t,\xi) \in [t_*, \tilde{T})\times \R^d} 
    |w(t,\xi)| \le M\eps^{1+\delta}  
 \Bigr\}.
\]
We observe that  
\[
i\pa_t w=\frac{\lambda}{t^a}
 \Bigl(|\eta_0+w|^{b}(\eta_0+w) - |\eta_0|^{b}\eta_0\Bigr)
 + \rho,
 \quad w(t_*,\xi)=\psi_1(\xi).
\]
We also note that $T_{**}>t_{*}$, because of the estimate 
\[
 |w(t_*,\xi)|=|\psi_1(\xi)| \le C_1\eps^{1+\delta}
 \le \frac{M}{2} \eps^{1+\delta}
\]
and the continuity of $w$. 
Now we set 
\[
f(t,\xi)= |w(t,\xi)|^2+\frac{C_2^2}{2C_3} \eps^{2+2\delta}.
\]
Then it follows from Lemma \ref{lem_elementary} that
\begin{align*}
 \pa_t f(t,\xi)
 =&
 2 \imagpart\Bigl(i\pa_t w \cdot \overline{w}\Bigr)\\
 \le&
 \frac{2|\lambda|}{t^a} (b+1) \Bigl(2|\eta_0| + |w| \Bigr)^{b} |w|^2
 +|\rho||w| \\
 \le& 
  \frac{2|\lambda|(b+1)}{t^a} \Bigl(2C_0 \eps + M\eps^{1+\delta} \Bigr)^{b} 
 |w|^2
 + |w|\cdot \frac{C_2\eps^{1+b+\delta}}{t^a}\\
 \le& 
  \frac{\eps^{b}}{t^a} \biggl\{\Bigl(C_3-\frac{1}{2}\Bigr) |w|^2
 + |w|\cdot C_2\eps^{1+\delta} \biggr\}\\ 
 \le& 
  \frac{\eps^{b}}{t^a} \biggl(C_3 |w|^2
 +\frac{C_2^2}{2} \eps^{2+2\delta} \biggr)\\ 
 =&
 \frac{C_3 \eps^{b}}{t^a} f(t,\xi)
\end{align*}
for $t\in (t_*, T_{**})$, as well as 
\[
 f(t_*,\xi)
 \le
  (C_1 \eps^{1+\delta})^2+ \frac{C_2^2}{2C_3} \eps^{2+2\delta}
 \le 
 \left(C_1^2+\frac{C_2^2}{2C_3}\right)\eps^{2+2\delta}.
\]
These lead to
\begin{align*}
 f(t,\xi) 
 &\le 
 f(t_*,\xi)
 \exp
 \left(
  \int_{t_*}^{\sigma \eps^{-2q}} \frac{C_3 \eps^{b}}{\tau^a}\, d\tau
 \right)\\
 &\le 
 \left(C_1^2+\frac{C_2^2}{2C_3}\right)\eps^{2+2\delta}
 \exp
 \left(
  \frac{C_3 \sigma^{1-a}}{1-a} \eps^{b-2q(1-a)}
  \right)\\ 
&\le
 \left(\frac{M}{2}\eps^{1+\delta}\right)^2,
\end{align*}
whence
\[
 |w(t,\xi)| \le \sqrt{f(t,\xi)} 
 \le
\frac{M}{2}\eps^{1+\delta}
\]
for $(t,\xi) \in [t_*, T_{**})\times \R^d$. 
This contradicts the definition of $T_{**}$ if $T_{**}$ is strictly less than 
$T_*$. Therefore we conclude that  $T_{**}=T_*$. In other words, we have 
\[
 \sup_{(t,\xi) \in [t_*, T_{*})\times \R^d}|w(t,\xi)| \le \sqrt{f(t,\xi)} 
 \le
 M\eps^{1+\delta}.
\]
Going back to the definition of $w$, we have
\[
 |\eta(t,\xi)| 
 \le 
  |\eta_0(t,\xi)| +|w(t,\xi)| 
 \le
  C_0\eps + M\eps^{1+\delta}
\]
for $(t,\xi) \in [t_*, T_{*})\times \R^d$, 
as desired.
\end{proof}

\section{Bootstrap argument in the large time}  
\label{sec_bootstrap}
Now we are ready to pursue the behavior of the solution 
$u(t)$ of \eqref{nls_d} for $t\gtrsim o(\eps^{-2\theta/d(1-\theta)})$.
For this purpose, we set $t_*=\eps^{-\theta/(1-\theta)d}$, and 
let $\eps$ be small enough to satisfy $\eps^{\theta/(1-\theta)d}< D_0$. 
Then, since $t_*\le D_0 \eps^{-2\theta/(1-\theta)d}$, 
Proposition \ref{prp_rough} gives us 
$E(t_*)\le 2\Phi \eps$. 
Next we set 
\[
 \tau_0:= 
 \left(\frac{(1-\theta)d}
 {\dis{2\theta \imagpart \lambda \sup_{\xi \in \R^d} 
   |\hat{\varphi}(\xi)|^{2\theta/d}}}\right)^{1/(1-\theta)}
\]
and fix $\sigma \in (0,\tau_0)$, 
$T \in (t_*, \sigma \eps^{-2\theta/d(1-\theta)}]$. 
Note that the right-hand side in \eqref{goal} is equal to $\tau_0^{1-\theta}$. 
For the solution $u(t)$ in the interval $t \in [0,T)$,  we put 
\[
 E(T)=\sup_{t\in [0,T)}  \|\op{U}(t)^{-1}u(t)\|_{\Sigma^{s}} 
\]
as in the proof of Proposition \ref{prp_rough}. 
The following lemma is the main step toward Theorem~\ref{thm_main}.

\begin{lem} \label{lem_apriori}
Let $\sigma$ and $T$  be as above.
Then there exist  constants $\eps_0>0$ and $K>4\Phi$, 
which are independent of $T$, such that
the estimate $E(T)\le K\eps$ 
implies the better estimate 
$E(T)\le K\eps/2$ if $\eps \in (0,\eps_0].$
\end{lem}

\begin{proof}
It  suffices to consider $t\in  [t_*,T)$, because we already know that 
$E(t_*)\le 2\Phi \eps$. For $t\in  [t_*,T)$, we set 
$A(t,\xi)=\op{F} \bigl[\op{U}(t)^{-1} u(t,\cdot)\bigr](\xi)$ 
and 
\[
 R(t,\xi)=\op{F} \bigl[\op{U}(t)^{-1} N(u(t,\cdot))\bigr](\xi)
 - t^{-\theta} N(A(t,\xi))
\]
so that 
\begin{align}
 i\pa_t A
 =
 \op{F} \op{U}(t)^{-1} \op{L}u
 =
 \op{F} \op{U}(t)^{-1} N(u)
 =
 \frac{\lambda}{t^{\theta}}|A|^{2\theta/d} A +R.
\label{reduced_ODE}
\end{align}
Next we take $\gamma=(2s-d)/8 \in (0,1/2]$. 
Note that $s-d/2=4\gamma >2\gamma$. 
Since $R$ can be written as
\[
 R(t,\xi)=\lambda\Bigl(
 \op{F} \op{U}(t)^{-1} G_{1+2\theta/d}(u)
 - t^{-\theta} G_{1+2\theta/d}(\op{F} \op{U}(t)^{-1}u)
 \Bigr),
\]
Lemma \ref{lem_remainder} yields 
\[
 |R(t,\xi)| 
\le 
 \frac{C}{t^{\theta+\gamma}}E(T)^{2\theta/d+1}
\le
 \frac{C \eps^{1+2\theta/d}}{t^{\theta}} K^{1+2\theta/d} t_*^{-\gamma}
\le
 \frac{C \eps^{1+2\theta/d+\gamma \theta/2d(1-\theta)}}{t^{\theta}}
\]
if $E(T)\le K\eps$ and 
$K^{1+2\theta/d} \eps^{\gamma\theta/2d(1-\theta)}\le 1$. 
Moreover, when we put 
$\psi(\xi)=A(t_*,\xi) -\eps \hat{\varphi}(\xi)$, we have 
\begin{align*}
|\psi(\xi)| 
&\le 
C\|\op{U}(t_*)^{-1}u(t_*,\cdot) - \eps \varphi\|_{H^{0,s}}
\\
&\le
C \int_0^{t_*} \|\op{U}(t)^{-1} N(u(t))\|_{H^{0,s}}\, dt\\
&\le
C\int_0^{\eps^{-\theta/(1-\theta)d}} 
 \frac{(2\Phi \eps)^{1+2\theta/d}}{(1+t)^{\theta}}\, dt\\
&\le
C\eps^{1+2\theta/d} 
\int_0^{\eps^{-\theta/(1-\theta)d}} \frac{dt}{(1+t)^{\theta}}\\
&\le
C\eps^{1+\theta/d},
\end{align*}
where we have used Lemma \ref{lem_composition}, Lemma \ref{lem_wSob} 
and Proposition \ref{prp_rough}. 
Therefore we can apply Lemma \ref{lem_ODE} 
with $\eta=A$, $a=\theta$, $b=2\theta/d$, 
$\delta=\min\{\theta/d, \gamma \theta/2d(1-\theta)\}$, 
$\psi_0=\hat{\varphi}$ and $\rho=R$ 
to obtain 
\[
 |A(t,\xi)|\le (C_{0}+1)\eps 
\]
for $(t,\xi) \in [t_*,T)\times \R^d$, where 
\[
C_0 =
 \frac{\sup_{\xi \in \R^d} |\hat{\varphi}(\xi)|}
 {(1-(\sigma/\tau_0)^{1-\theta})^{d/2\theta}}.
\]
Note that $C_0$ is independent of $\eps$, $K$ and $T$. 
By this estimate and Lemma \ref{lem_Linfty}, we have 
\begin{align*}
 \|u(t)\|_{L^{\infty}}
 &\le
 t^{-d/2}\|A(t,\cdot)\|_{L^{\infty}}
 +Ct^{-d/2-\gamma} \|\op{U}(t)^{-1}u(t)\|_{\Sigma^{s}}\\
 &\le 
 t^{-d/2} \Bigl( (C_0+1)\eps
 +CK\eps t_*^{-\gamma}\Bigr)\\
 &\le
  t^{-d/2} \Bigl(C\eps + CK\eps^{1+\gamma\theta/d(1-\theta)}\Bigr)\\
 &\le 
 C\eps t^{-d/2}, 
\end{align*}
if $K\eps^{\gamma\theta/d(1-\theta)}\le 1$. 
By the standard energy inequality combined with Lemma \ref{lem_composition}, 
we obtain 
\begin{align*}
\sup_{t_*\le t< T} \|\op{U}(t)^{-1}u(t)\|_{\Sigma^{s}}
 &\le
   \|\op{U}(t_*)^{-1}u(t_*)\|_{\Sigma^{s}} 
 \exp\left(\int_{t_*}^{T}C\|u(t)\|_{L^{\infty}}^{2\theta/d}dt \right)\\
 &\le
 2\Phi \eps 
 \exp\left(C\eps^{2\theta/d} 
 \int_{0}^{\sigma \eps^{-2\theta/d(1-\theta)}}\frac{dt}{t^\theta} 
\right)\\
 &\le
  \bigl(2\Phi e^{C_{\star}} \bigr) \eps
\end{align*}
for $t\in [t_*,T)$, where the constant $C_{\star}$ is independent of $\eps$, 
$K$ and $T$. 
Now we set $K=4\Phi e^{C_{\star}}$. Then we arrive at the desired estimate 
$E(T) \le K\eps /2$.
\end{proof}

\noindent
{\em Proof of Theorem \ref{thm_main}.}\ 
Let $T_{\eps}$ be the lifespan defined in the statement of 
Theorem \ref{thm_main}. 
We fix $\sigma \in (0, \tau_0)$ and set 
\[
 T^*=\sup\bigl\{ t \in [0,T_{\eps})\, \bigm|\, E(t)\le K\eps \bigr\},
\]
where $K$ is given in Lemma \ref{lem_apriori}. 
Now we assume 
$T^* \le \sigma \eps^{-2\theta/d(1-\theta)}$. 
Then, Lemma \ref{lem_apriori} 
with $T=T^*$ implies $E(T^*)\le K\eps/2$ if $\eps\le \eps_0$. 
By the continuity of $[0,T_{\eps})\ni T\mapsto E(T)$, we can 
choose $\tilde{\delta}>0$ such that $E(T^*+\tilde{\delta})\le K\eps$, 
which contradicts the definition of $T^*$. Therefore we must have 
$T^* \ge \sigma \eps^{-2\theta/d(1-\theta)}$ if $\eps\le \eps_0$. 
As a consequence, we obtain
\[
 \liminf_{\eps \to +0} \eps^{2\theta/d} {T_{\eps}}^{1-\theta}
 \ge 
 \sigma^{1-\theta}.
\]
Since $\sigma \in (0, \tau_0)$ is arbitrary, we arrive at the 
desired estimate \eqref{goal}.
\qed

\section{The critical case}  \label{sec_critical}
We conclude this paper with a few comments on the critical case $\theta=1$, 
that is, 
\begin{align}
\left\{\begin{array}{cl}
 \dis{i\pa_t u+\frac{1}{2}\Delta u=\lambda |u|^{2/d}u}, & t>0, \ x\in \R^d,\\
 u(0,x)=\eps \varphi(x), &x\in \R^d,
\end{array}\right.
\label{nls_c}
\end{align}
with $\imagpart \lambda>0$. 
As mentioned in the introduction, 
one dimensional case ($d=1$) has been covered in the previous works 
\cite{Su}, \cite{SaSu}. 
Minor modifications of the method in the previous sections 
allow us to treat the case of $d=2$, $3$. 

\begin{thm} \label{thm_critical}
Let $1\le d\le 3$ and $\lambda\in \C$ with $\imagpart \lambda >0$. 
Assume $\varphi \in \Sigma^{s}$ with $s$ satisfying \eqref{index}. 
Let $T_{\eps}$ be the supremum of $T>0$ such that \eqref{nls_c} admits a 
unique solution  $u$ satisfying $\op{U}(\cdot)^{-1}u\in C([0,T); \Sigma^{s})$. 
Then we have
\[
 \liminf_{\eps\to +0} \left( \eps^{2/d} \log T_{\eps} \right)
 \ge 
 \frac{d}
 {\dis{2 \imagpart \lambda \sup_{\xi \in \R^d}|\hat{\varphi}(\xi)|^{2/d} }}.
\]
\end{thm}

Since the proof is almost the same as that for Theorem \ref{thm_main}, 
we only point out where this lower bound comes from. 
As in Section \ref{sec_bootstrap}, we can see that 
$A(t,\xi)=\op{F} \bigl[\op{U}(t)^{-1} u(t,\cdot)\bigr](\xi)$ satisfies 
\begin{align*}
 i\pa_t A
 =
 \frac{\lambda}{t}|A|^{2/d} A +R
\end{align*}
with $A(1,\xi)=\eps \hat{\varphi}(\xi)+\psi(\xi)$,
where $R$ and $\psi$ are regarded as remainder terms. 
If $R$ and $\psi$ could be neglected, then we would have
\[
 |A(t,\xi)|^{2/d}
 =
 \frac{(\eps |\hat{\varphi}(\xi)|)^{2/d}}
    {1  - (2/d) \imagpart \lambda (\eps|\hat{\varphi}(\xi)|)^{2/d} \log t}.
\]
The desired lower bound is characterized by the time 
when this denominator vanishes.

\medskip
\subsection*{Acknowledgments}
The authors are grateful to Professors 
Naoyasu Kita and Hironobu Sasaki for their useful conversations 
on this subject. 
The work of the second author (H.~S.) is supported by 
Grant-in-Aid for Scientific Research (C) (No.~25400161), JSPS.


\end{document}